\documentclass[a4paper,12pt,oneside,reqno]{amsart}

\usepackage[T1]{fontenc}
\usepackage[utf8]{inputenc}
\usepackage{lmodern}
\usepackage[english]{babel}

\usepackage[%
left=2.5cm,       
right=2.5cm,      
top=3.5cm,        
bottom=3.5cm,     
heightrounded,    
bindingoffset=0mm 
]{geometry}
\usepackage{amsrefs}
\usepackage{amsmath}
\usepackage{amssymb}
\usepackage{mathtools}
\usepackage{mathrsfs}
\usepackage[colorlinks=true,linkcolor=blue,urlcolor=blue,citecolor=blue]{hyperref}

\numberwithin{equation}{section}

\usepackage{hyperref} 

\usepackage{tikz}
\usepackage{graphicx}
\usepackage{xcolor}
\usepackage[font=small]{caption}

\usepackage{bookmark}

\theoremstyle{plain}
\newtheorem{theorem}{Theorem}[section]
\newtheorem{lemma}[theorem]{Lemma}
\newtheorem{proposition}[theorem]{Proposition}

\newtheorem*{mtheorem}{Main Theorem}
\theoremstyle{definition}

\newtheorem{remark}[theorem]{Remark}
\newtheorem{open.problem}[theorem]{Open Problem}

\newcommand{\Leb}[1]{\mathcal{L}^{#1}} 

\newcommand{\N}{\mathbb{N}}
\newcommand{\R}{\mathbb{R}}

\allowdisplaybreaks

\title[Stability of the Gaussian Faber-Krahn inequality]{Stability of the Gaussian Faber-Krahn inequality}
\author[A.\ Carbotti]{Alessandro Carbotti}
\address{Dipartimento di Matematica
	e Fisica ``E. De Giorgi'', Universit\`a del Salento,
	Via Per Arnesano, 73100 Lecce, Italy.}
\email{alessandro.carbotti@unisalento.it}
\author[S.\ Cito]{Simone Cito}
\address{Dipartimento di Matematica
	e Fisica ``E. De Giorgi'', Universit\`a del Salento,
	Via Per Arnesano, 73100 Lecce, Italy.}
\email{simone.cito@unisalento.it}
\author[D. A. \ La Manna]{Domenico Angelo La Manna}
\address{Dipartimento di Matematica e Applicazioni “R. Caccioppoli”, Università degli Studi di Napoli ``Federico II'', Via Cintia, Monte S. Angelo, 80126 Naples, Italy}
\email{domenicoangelo.lamanna@unina.it}
\author[D.\ Pallara]{Diego Pallara}
\address{Dipartimento di Matematica
	e Fisica ``E. De Giorgi'', Universit\`a del Salento, and INFN, Sezione di Lecce,
	Via Per Arnesano, 73100 Lecce, Italy.}
\email{diego.pallara@unisalento.it}

\date{\today}  



\begin{document}
	\begin{abstract}
		We prove a quantitative version of the Gaussian Faber-Krahn type inequali\-ty proved in \cite{BeChFe} for the first Dirichlet eigenvalue of the Ornstein-Uhlenbeck operator, estima\-ting the deficit in terms of the Gaussian Fraenkel asymmetry. As expected, the multiplicative constant only depends on the prescribed Gaussian measure.
	\end{abstract}
	
	\maketitle
	
	\tableofcontents
	
	\section{Introduction}
	\label{sec:intro}
	In the plethora of inequalities studied in shape optimization the Faber-Krahn type ones are classical issues: given a measure $\nu$ and a second order elliptic operator $L$ in divergence form in $L^2(\R^N;\nu)$, among all $\nu$-measurable sets $\Omega$ with fixed finite measure, there exists, up to some group of transformations, a unique set $\Omega_{\text{opt}}$ that minimizes the first Dirichlet eigenvalue $\lambda_L(\Omega)$ of a given domain $\Omega$. Namely,
	\begin{equation}
	\label{eq:introqual}
	D_L(\Omega):=\lambda_L(\Omega)-\lambda_L(\Omega_{\text{opt}})\ge 0,\qquad \nu(\Omega)=\nu(\Omega_{\text{opt}}).
	\end{equation}
	  Once the optimal set has been identified, one can try to prove the stability of inequality \eqref{eq:introqual} by quantifying how far a set is from being optimal for $\lambda_L$ in terms of some geometric asymmetry index $d(\Omega)$. More precisely, a quantitative enhancement of \eqref{eq:introqual} is
	\begin{equation}
	\label{eq:introquant}
	D_L(\Omega)\ge CG(d(\Omega)),
	\end{equation}
	where $C>0$ is a constant and $G:[0,+\infty)\rightarrow[0,+\infty)$ is some modulus of continuity.
	The classical works by Faber \cite{Faber} and Krahn \cite{Krahn} prove that if $\nu=\mathcal{L}^N$, $L=-\Delta$ and $\Omega$ is bounded then $\Omega_{\text{opt}}=B_R$ for $R=\left(\frac{\mathcal{L}^N(\Omega)}{\omega_N}\right)^{1/N}$.
	The study of the stability of the Faber-Krahn inequality for the first eigenvalue of the Dirichlet Laplacian started with the pioneering works \cite{HanNad, Melas}. The case in which the asymmetry index $d(\Omega)$ is the Fraenkel asymmetry $\mathcal{A}(\Omega):=\inf_{x\in\R^N}\frac{\mathcal{L}^N(\Omega\triangle B_R(x))}{\mathcal{L}^N(\Omega)}$ is a consequence of \cite[Theorem 2.1]{Bhatt} in the case $N=2$ and \cite[Theorem 1.1]{FuMaPr2} in the general case, with $G(r)=r^3$ and $G(r)=r^4$, respectively. Nevertheless, it had already been conjectured independently in \cite{BhaWei} and \cite{nadir} that the inequality should be true with $G(r)=r^2$, which is the expected sharpest power in inequalities like \eqref{eq:introquant} when $d(\Omega)=\mathcal{A}(\Omega)$
	. Actually, the stability of the Faber-Krahn inequality with $G(r)=r^2$ has been proved in \cite{BrDeVe} using the techniques developed in \cite{AcFuMo, CicLeo}. The sharpness of the quadratic power for the Faber-Krahn inequality when $d(\Omega)=\mathcal{A}(\Omega)$ is a known fact, see for instance \cite{BraDeP, BrDeVe, Fusco}. 
	When $\nu$ is the Gaussian measure $\gamma$ and $L$ is the Ornstein-Uhlenbeck operator $-\Delta_\gamma$ it is proved in \cite{BeChFe} that \eqref{eq:introqual} holds true with
	\[
	\Omega_{\text{opt}} = H_{\omega,r} = 
	\left\{x\in \mathbb{R}^N\quad\text{s.t.}\quad x\cdot\omega <r \right\},
	\]
	for some $\omega\in\mathbb{S}^{N-1}$ and for $r\in\R$ uniquely determined such that $\gamma(H_{\omega,r})=\gamma(\Omega)$.
A key tool used to prove optimality of halfspaces in the Gaussian setting is the notion of Ehrhard symmetrization introduced in \cite{EhrScand}. We notice that qualitative spectral inequalities in the Gaussian framework in which the optimal shape is the halfspace are also proved in \cite{ChiDib, ChiGav} under other boundary conditions. We finally point out that a wide class of quantitative weighted isoperimetric inequalities has been treated in \cite{FusLam}, in which the authors consider a class of log-convex weights that does not include the Gaussian one.
	
	The goal of this paper is to prove the quantitative inequality \eqref{eq:introquant} with $L=-\Delta_\gamma$, $G(r)=r^3$ and choosing as $d(\Omega)$ the Gaussian Fraenkel asymmetry. Nevertheless we conjecture that also in the Gaussian setting the power $3$ of the Fraenkel asymmetry can be replaced with the sharpest power $2$ as for the Gaussian perimeter (see \cite{BarBraJul}).

From now on, to simplify the notation we set $\lambda_\gamma=\lambda_{-\Delta_\gamma}$ and $D_\gamma=D_{-\Delta_\gamma}$. 
	
In order to state the Main Theorem, we introduce the Gaussian Fraenkel asymmetry of an open set $\Omega$, defined as
$$
\mathcal{A}_\gamma(\Omega):=\min_{\omega\in\mathbb{S}^{N-1}}\frac{\gamma(\Omega\triangle H_{\omega,r})}{\gamma(\Omega)},
$$
where the halfspaces
$$
H_{\omega,r}:=\left\{x\in \mathbb{R}^N\quad\text{s.t.}\quad x\cdot\omega <r \right\}
$$
have the same Gaussian measure of $\Omega$.

\begin{mtheorem}
	\label{th:maintheorem}
	Let $N\geq 1$ and $m\in (0,1)$. For any open set $\Omega$ with $\gamma(\Omega)=m$ we have
	\begin{equation} \label{eq:main}
	D_\gamma (\Omega):=\lambda_\gamma (\Omega)- \lambda_\gamma(H) \geq  C_{m} \mathcal{A}_\gamma(\Omega)^3,
	\end{equation}
	where $H$ is any halfspace with $\gamma(H)=\gamma (\Omega)$ and $C_{m}$ is a positive constant which depends only on $m$.
\end{mtheorem}	
	
	Inequalities of isoperimetric type in the Gaussian setting have been proved in \cite{borell, CarKer, ehrhard, SudCir}, in \cite{AmbMon} in the nonsmooth context of $\text{RCD}(K,\infty)$ spaces that generalize the Gauss space as metric measure spaces, and in \cite{NovPalSir} for a fractional perimeter in the infinite-dimensional setting of abstract Wiener spaces, while the stability has been faced in \cite{BarBraJul, CiFuMaPr, JulSar} and in \cite{CCLP} also in the fractional setting. See Section \ref{sec:sezdue} for all the missing definitions.
	
 The paper is organized as follows:
 in Section \ref{sec:sezdue}, after introducing some notation,  we recall some properties of eigenvalues and eigenfunction of the Dirichlet-Ornstein Uhlenbeck operator (Subsection \ref{sec:subdue}) and we prove that the Gaussian Faber-Krahn profile enjoys some useful regularity properties (Subsection \ref{sec:subbdue}). In Section \ref{sec:seztre} we delve into the proof of our Main Theorem.  
 
We follow the strategy introduced by Hansen and Nadirashvili in \cite{HanNad}. We exploit a quantitative version of the P\'{o}lya-Szeg\"{o} inequality in the Gaussian framework joint with the sharp quantitative isoperimetric inequality proved in \cite{BarBraJul} to control the propagation of the asymmetry of the level sets (see Proposition \ref{prop:quattroquattronostra}). 

 We notice that the techniques in the proof of our Main Theorem seem to be flexible enough to be used in the fractional context through an extension procedure \`{a} la Caffarelli-Silvestre as in \cite{BraCinVit, CCLP}. We also point out that in \cite{BrDeVe} the stability for the scale invariant functional
$$
F(\Omega):=|\Omega|^{2/N}\lambda_{-\Delta}(\Omega)
$$
has been proved. Since the function $t\mapsto t^{-2/N}$ is exactly the Faber-Krahn profile for the first eigenvalue of the Dirichlet Laplacian, in the same vein we can state our stability result for the functional
$$
F_\gamma(\Omega):=\frac{\lambda_\gamma(\Omega)}{g(\gamma(\Omega))}
$$
even though in the Gaussian framework the scale invariance of $F_\gamma$ does not hold. Here, setting 
$$
\Phi(r):=\frac{1}{\sqrt{2\pi}}\int_{-\infty}^r e^{-\frac{t^2}{2}}\: dt,\qquad r\in\R,
$$
we define $g(m):=\lambda_\gamma(H_{\omega,\Phi^{-1}(m)})$, see Section \ref{sec:sezdue}.

\paragraph*{\bf Acknowledgements}  
The authors are member of GNAMPA of the Istituto Nazionale di Alta Matematica (INdAM). 
A.C. and S.C., D.A.L., D.P. respectively acknowledge the support of the INdAM - GNAMPA 2023 Projects ``Problemi variazionali per funzionali e operatori non-locali'', ``Disuguaglianze isoperimetriche e spettrali'', ``Equazioni differenziali stocastiche e operatori di Kolmogorov in dimensione infinita''. A.C., S.C. and D.A.L. acknowledge the support of the INdAM - GNAMPA 2024 Project ``Ottimizzazione e disuguaglianze funzionali per problemi geometrico-spettrali locali e nonlocali''.
A.C., S.C. and D.P. have been also partially supported by the PRIN 2022 project 20223L2NWK. D.A.L. has been also partially supported by the PRIN 2022 project 2022E9CF89.

\section{Notation and preliminary results}
\label{sec:sezdue}

For $N\in\N$ we denote by $\gamma_N$ and $\mathcal{H}^{N-1}_\gamma$ the Gaussian measure on $\R^N$ and the 
$(N-1)$-Hausdorff Gaussian measure
\begin{align*}
\gamma_N&:=\frac{1}{(2\pi)^{N/2}}e^{-\frac{|\cdot|^2}{2}}\Leb{N},
\\
\mathcal{H}^{N-1}_\gamma&:=\frac{1}{(2\pi)^{(N-1)/2}}e^{-\frac{|\cdot|^2}{2}}\mathcal{H}^{N-1},
\end{align*}
where $\Leb{N}$ and $\mathcal{H}^{N-1}$ are the Lebesgue measure and the Euclidean $(N-1)$-dimensional Hausdorff measure, respectively. 
When $k\in \{1,\dots, N\}$ is a given integer, we denote by $\gamma_k$ the standard $k$-dimensional Gaussian measure in $\R^k$; when there is no ambiguity we simply write $\gamma$ instead of $\gamma_N$. 

The Gaussian perimeter of a measurable set $E$ in an open set $\Omega$ is defined as
$$P_\gamma(E;\Omega)=\sqrt{2\pi}\:\sup\left\{\int_E\left(\text{div}\,\varphi-\varphi\cdot x\right)\:d\gamma(x):\varphi\in C^\infty_c(\Omega;\R^N),\ \|\varphi\|_\infty\le 1\right\}.$$
If $\Omega=\R^N$, we denote the Gaussian perimeter of $E$ in the whole $\R^N$ simply by $P_\gamma(E)$. Moreover, if $E$ has finite Gaussian perimeter, then $E$ has locally finite Euclidean perimeter and it holds
$$
P_\gamma(E)=\mathcal{H}^{N-1}_\gamma(\partial^\ast E)=\frac{1}{(2\pi)^{\frac{(N-1)}{2}}}
\int_{\partial^\ast E}e^{-\frac{|x|^2}{2}}d\mathcal{H}^{N-1}(x),
$$
where $\partial^\ast E$ is the reduced boundary of $E$. We refer to \cite{AFP} for the properties of 
sets with finite perimeter. 

We introduce the strictly increasing function $\Phi:\mathbb{R}\rightarrow (0,1)$ by 
$$
\Phi(r):=\int_{-\infty}^r d\gamma_1(t),
$$
and its inverse $\Phi^{-1}:(0,1)\rightarrow \mathbb{R}$.
Defining, for $\omega\in\mathbb{S}^{N-1}$ and $r\in\mathbb{R}$, $H_{\omega,r}$ the halfspace
$$
H_{\omega,r}:=\left\{x\in \mathbb{R}^N\quad\text{s.t.}\quad x\cdot\omega <r \right\},
$$ 
we have 
$$
\gamma(H_{\omega,r})=\Phi(r)
$$
and
$$
P_\gamma(H_{\omega,r})=e^{-r^2/2}.
$$
Moreover, the Gaussian perimeter of any halfspace with Gaussian volume $m\in (0,1)$ is given by
\begin{equation}
\label{eq:isoperimetricfunction}
I(m):=e^{-\frac{\Phi^{-1}(m)^2}{2}},
\end{equation}
where $I:(0,1)\rightarrow(0,1)$ is usually called \emph{isoperimetric function}. The Gaussian isoperimetric inequality reads 
\begin{equation}
\label{eq:gaussisopine}
P_\gamma(E) \ge I(\gamma(E)),
\end{equation}
and halfspaces are the unique (see \cite{CarKer}) volume constrained minimizers of the Gaussian perimeter. A sharp stability result for \eqref{eq:gaussisopine} has been obtained in \cite{BarBraJul} and it reads 
	\begin{equation}
	\label{eq:cifumapr}
	P_\gamma(E) - I(\gamma(E))= P_\gamma(E)-e^{-\frac{r^2}{2}} \ge 
	\frac{e^{\frac{r^2}{2}}}{4c(1+r^2)} \mathcal{A}_\gamma(E)^2,
	\end{equation}
for any set $E$ such that $\gamma(E)=m=\Phi(r)$ and for some absolute constant $c>0$.

Following \cite{EhrScand}, we introduce a suitable notion of symmetrization in the Gauss space. 
First, for any $J\subset\R$ we set 
\begin{equation}\label{defI*}
J^*=(-\infty, \Phi^{-1}(\gamma_1(J))) .
\end{equation}
Then, for $h\in \R^N$ with $|h|=1$, we consider the projection
$x'=x-(x\cdot h) h$ and write $x=x'+th$ with $t\in \R$,  
and for every measurable function $u:\R^N\to\R$ we define the symmetrized function in the sense of Ehrhard
\begin{equation}\label{defu*}
u_h^*(x'+th)=\sup\Bigl\{c\in {\mathbb R}:\ t\in\{u(x',\cdot)>c\}^*\Bigr\}. 
\end{equation}
The Gaussian rearrangement of a set is a set with the same measure whose sections in the direction $h$ are halflines, and the superlevel sets of the rearrangement $u^*$ of a function $u$ with respect to a direction $h$ have the same shape. 
Notice that if $u$ is (weakly) differentiable, $u^*_h$ is (weakly) differentiable as well and the inequality
\[
\int_{\R^N} |\nabla u^*_h(x)|^2\, d\gamma(x) \leq \int_{\R^N}|\nabla u(x)|^2\, d\gamma(x)
\]
holds, see \cite[Theorem 3.1]{ehrhard} for the Lipschitz case; the Sobolev case easily follows by approximation. 
Since symmetrization preserves the class of characteristic functions, for every measu\-rable set $\Omega\subset \R^N$
we may define the Ehrhard-symmetrized set $\Omega_h^*$ through the equality 
\[
\chi_{\Omega_h^*}=(\chi_\Omega)_h^* .
\]
We define the {\em Gaussian Fraenkel asymmetry} and the {\em Gaussian Faber-Krahn deficit} of 
a set $\Omega$ as
$$
\mathcal{A}_\gamma(\Omega):=\min_{\omega\in\mathbb{S}^{N-1}}\frac{\gamma(\Omega\triangle H_{\omega,r})}{\gamma(\Omega)},
$$
and
$$
D_{\gamma}(\Omega):=\lambda_\gamma(\Omega)-\lambda_{\gamma}(H_{\omega,r}),
$$
where $\triangle$ stands for the symmetric difference, $\lambda_\gamma(\Omega)$ is the 
{\em first Dirichlet eigenvalue of the Ornstein-Uhlenbeck operator} with respect to the domain $\Omega$, see Subsection \ref{sec:subdue}, and $r=\Phi^{-1}(\gamma(\Omega))$. 
These definitions are motivated 
by the fact that halfspaces are the optimal sets for the Gaussian Faber-Krahn problem as well, see 
\cite{BeChFe}. In particular, we can rephrase the statement of \cite[Theorem 3.1]{BeChFe} without assuming the volume constraint by stating that for any measurable set it holds that
\begin{equation}
\label{eq:analyticfk}
\frac{\lambda_\gamma(\Omega)}{g(\gamma(\Omega))}\ge \frac{\lambda_\gamma(H_{\omega,r})}{g(\gamma(H_{\omega,r}))}=1,
\end{equation}
where the function $g:[0,1)\rightarrow[0,+\infty)$ defined by
\[ 
g(m)=\lambda_\gamma(H_{\omega,\Phi^{-1}(m)}) 
\]
is nonnegative and strictly decreasing, see \cite{ehrhard}. In particular for any measurable set $\Omega$ we have that $\lambda_\gamma(\Omega)\ge g(\gamma(\Omega))$ and the equality holds if and only if $\Omega=H_{\omega,r}$ for some $\omega\in\mathbb{S}^{N-1}$ and $r$ such that $\gamma(H_{\omega,r})=\gamma(\Omega)$. From now on we refer to the function $g$ as the \textit{Gaussian Faber-Krahn profile}.

We recall that in the Gaussian case the Ornstein-Uhlenbeck operator $\Delta_\gamma$ defined
for $u$ sufficiently smooth as
$$
(\Delta_\gamma u)(x):=(\Delta u)(x)- x\cdot\nabla u(x),
$$
plays in the Gaussian setting the same 
role as the Laplacian in the Euclidean one.

\subsection{Properties of eigenvalues and eigenfunctions of $-\Delta_\gamma$}\label{sec:subdue}

In the sequel we denote $H^1(\Omega,\gamma)$ the subspace of the functions $u\in L^2(\R^N,\gamma)$ such that $\|\nabla u\|_{L^2(\Omega,\gamma)}$ is finite, and we denote by $H^1_0(\Omega,\gamma)$ the completion of $C^\infty_c(\Omega)$ with respect to this norm (notice that $\|\nabla \cdot\|_{L^2(\Omega,\gamma)}$ is actually a norm in $C_c^\infty(\Omega)$).

The first Dirichlet eigenvalue of the Ornstein - Uhlenbeck (or, briefly, the first Gaussian Dirichlet eigenvalue) is the smallest real number $\lambda$ such that
\begin{equation}
\label{eq:pdeOUeigen}
\begin{cases}
-\Delta_\gamma u=\lambda u\quad\text{in}\quad\Omega \\
u=0\quad\text{on}\quad\partial\Omega
\end{cases}
\end{equation}
admits a nontrivial solution in $H^1_0(\Omega,\gamma)$. From now on we denote such eigenvalue by $\lambda_\gamma(\Omega)$, and we call any nontrivial solution of \eqref{eq:pdeOUeigen} a {\em first eigenfunction of $\Omega$}.

We notice that \eqref{eq:pdeOUeigen} has a variational formulation. Indeed, any weak solution of \eqref{eq:pdeOUeigen} verifies

\begin{equation}
\label{eq:variationalformulation}
\int_{\Omega}\nabla u \cdot \nabla \varphi\: d\gamma=\lambda\int_\Omega u\varphi\:d\gamma,
\end{equation}
for any $\varphi\in H^1_0(\Omega,\gamma)$.

Therefore, it is not difficult to see that $\lambda_\gamma(\Omega)$ admits the following characterization
\begin{equation}\label{eq:1eigen}
\lambda_\gamma(\Omega)=\min_{u\in H^1_0(\Omega,\gamma)}\frac{\displaystyle\int_{\Omega}|\nabla u|^2\:d\gamma}{\displaystyle\int_{\Omega}u^2\:d\gamma}=\min_{\substack{u\in H^1_0(\Omega,\gamma) \\ \|u\|_{L^2(\Omega,\gamma)}=1}}\int_{\Omega}|\nabla u|^2\:d\gamma,
\end{equation}
and the minimum is achieved on any eigenfunction $u_\Omega$. 


Moreover, by standard spectral theory the eigenvalues of $-\Delta_\gamma$ form an increasing sequence
$$
0<\lambda_{\gamma,1}:=\lambda_\gamma\le\lambda_{\gamma,2}\le\cdots\le\lambda_{\gamma,k}\le\lambda_{\gamma,k+1}\le\cdots,
$$
with $\lambda_{\gamma,k}\rightarrow +\infty$ as $k\to +\infty$,
and for any $k\in\N$, $\lambda_{\gamma,k}$ has the following variational characterization
$$
\lambda_{\gamma,k}(\Omega)=\min_{u\in\mathbb{P}^k}\frac{\displaystyle\int_{\Omega}|\nabla u|^2\:d\gamma}{\displaystyle\int_{\Omega}u^2\:d\gamma}=\min_{\substack{u\in \mathbb{P}^k \\ \|u\|_{L^2(\Omega,\gamma)}=1}}\int_{\Omega}|\nabla u|^2\:d\gamma
$$
where
$$
\mathbb{P}^k:=\left\{u\in H^1_0(\Omega,\gamma)\quad\text{s.t.}\quad\left\langle u,u_{\Omega,j}\right\rangle=0\quad\forall j=1,\ldots,k-1\right\},
$$
and the minimum is attained in $u=u_{\Omega,k}$ where we have set $u_{\Omega,1}:=u_\Omega$.

The next Lemma is very classical and provides some useful properties of the first Dirichlet eigenvalue and eigenfunction of $-\Delta_\gamma$.

\begin{lemma}
Let $\Omega\subset\R^N$ be an open connected set with $\gamma(\Omega)<1$. Then, we have that
	\begin{enumerate}
		\item the first eigenfunction $u_\Omega$ is analytic and it does not change sign in $\overline{\Omega}$;
		\item the first eigenvalue $\lambda_\gamma(\Omega)$ is simple.
	\end{enumerate}
\end{lemma}

\begin{remark}
	By the analyticity of $u_\Omega$ it follows that the function $t\mapsto\gamma\left(\left\{u_\Omega>t\right\}\right)$ is absolutely continuous and $\partial^\ast\left\{u_\Omega>t\right\}=\partial\left\{u_\Omega>t\right\}=\{u_\Omega=t\}$.
\end{remark}

\subsection{Local bilipschitz continuity of the Faber-Krahn profile}
\label{sec:subbdue} 

We now prove a regularity result for $g$ that is crucial in the proof of our Main Theorem. To do this we quote the following technical result from 
\cite{Lieb}, see Theorem 1.13 and Corollary 1.15.
\begin{theorem} \label{th:lieb}
Let $V:\R\to\R$ be convex, let $C_0$, $C_1$ two nonempty intervals and $C_\tau:=\tau C_1+(1-\tau)C_0$, $\tau\in[0,1]$. If $\lambda(\tau)$ is the first  Dirichlet eigenvalue of the Schr\"odinger operator $\mathcal{H}_V:=-D^2+V$ on $C_\tau$, namely
	$$
	\begin{cases}
	\mathcal{H}_Vw=\lambda(\tau) w\quad\text{in}\quad C_\tau \\
	w=0\quad\text{in}\quad\partial C_\tau,
	\end{cases}
	$$
	then $\lambda$ is a convex function with respect to $\tau\in[0,1]$.
\end{theorem}
We are now ready to prove the following
\begin{proposition}
	\label{prop:bilipschitz}
	The Gaussian Faber-Krahn profile $g$ is invertible and locally bilipschitz continuous.
\end{proposition}
\begin{proof}
We start by proving that $g$ is locally Lipschitz continuous. Let $r\in\R$, let $H_r=\{x\in\R^N:\ x_N<r\}$ and let $u_r$ be the solution of
$$
\begin{cases}
-\Delta w+x\cdot \nabla w=\lambda_\gamma(H_r)w\quad\text{in}\ H_r\\
w=0\quad\text{on}\ \partial H_r,
\end{cases}
$$
with $\left\|u_r\right\|_{L^2(H_r,\gamma)}=1$, i.e., $u_r$ is a normalized first eigenfunction relative to $H_r$. Since $u_r$ only depends on $x_N$, we are reduced to the one dimensional case and we may consider $u_r:(-\infty,r]\to[0,+\infty)$ as the solution of
$$
\begin{cases}
-w''(x_N)+x_N w'(x_N)=\lambda_\gamma(H_r)w(x_N)\quad\text{in}\ (-\infty,r)\\
w(r)=0,
\end{cases}
$$
with $\left\|u_r\right\|_{L^2((-\infty,r),\gamma_1)}=1$ so that 
$$
\lambda_\gamma(H_r)=\int_{-\infty}^r|u_r'(x_N)|^2d\gamma_1(x_N). 
$$
For any $h>0$ we set
$$
v_{r,h}(x_N):=u_r(x_N+h)e^{-\frac{x_Nh}{2}}e^{-\frac{h^2}{4}}.
$$
It is easily seen that $\left\|v_{r,h}\right\|_{L^2((-\infty,r-h),\gamma_1)}=1$ for any $h>0$ and
$$
v'_{r,h}(x_N)=u'(x_N+h)e^{-\frac{x_Nh}{2}}e^{-\frac{h^2}{4}}-\frac h2v_{r,h}(x_N).
$$
Using the decreasing monotonicity of $\lambda_\gamma$ with respect to the set inclusion and the variational characterization of $\lambda_\gamma(H_{r-h})$ we get
\begin{align*}\allowdisplaybreaks
\lambda_\gamma(H_r)\leq&\lambda_\gamma(H_{r-h})\le\left\|v'_{r,h}\right\|^2_{L^2((-\infty,r-h),\gamma_1)}
\\
=&e^{-\frac{h^2}{2}}\int_{-\infty}^{r-h}|u'_r(x_N+h)|^2e^{-x_Nh}d\gamma_1(x_N) 
\\
&-he^{-\frac{h^2}{4}}\int_{-\infty}^{r-h}u'_r(x_N+h)e^{-\frac{x_Nh}{2}}v_{r,h}(x_N)d\gamma_1(x_N)+\frac{h^2}{4}
\\
=&\int_{-\infty}^{r-h}|u'_r(x_N+h)|^2\gamma_1(x_N+h)dx_N
\\
&-h\int_{-\infty}^{r-h}u_r(x_N+h)u'_r(x_N+h)\gamma_1(x_N+h)dx_N
+\frac{h^2}{4}
\\
\le&\lambda_\gamma(H_r)+h\left(\int_{-\infty}^ru^2_r(x_N)d\gamma_1(x_N)\right)^{1/2}\left(\int_{-\infty}^r|u'_r(x_N)|^2d\gamma_1(x_N)\right)^{1/2}+\frac{h^2}{4}
\\
=&\lambda_\gamma(H_r)+h\sqrt{\lambda_\gamma(H_r)}+\frac{h^2}{4}.
\end{align*}
Therefore for any $h>0$ we have 
$$
0\le\frac{\lambda_\gamma(H_{r-h})-\lambda_\gamma(H_{r})}{h}\le\sqrt{\lambda_\gamma(H_r)}+\frac{h}{4}.
$$
Since the function $\Lambda(r):=\lambda_\gamma(H_r)$ is strictly monotone then $\Lambda$ is a.e. differentiable in the whole of $\R$ and 
$$
|\Lambda'(r)|\le\sqrt{\Lambda(r)}\quad \text{for a.e. }r\in\R.
$$
By using optimality of the halfspace for $\lambda_\gamma$ we have that
$$
\Lambda(r)=\lambda_\gamma(H_r)=g(\gamma(H_r))=g(\Phi(r))
$$
therefore $g=\Lambda\circ\Phi^{-1}$ and it is locally Lipschitz continuous being the composition of two locally Lipschitz continuous functions.

Now, to prove that also $g^{-1}$ is locally Lipschitz, we make use of Theorem \ref{th:lieb}. 
If we set $v_r(\varrho):=\frac{e^{-\frac{\varrho^2}{4}}}{(2\pi)^{1/4}}u_r(\varrho)$, $\varrho\le r$, we have 
$\left\|v_r\right\|_{L^2(-\infty,r)}=\left\|u_r\right\|_{L^2((-\infty,r),\gamma_1)}=1$. 
Moreover $v_r$ solves
$$
\begin{cases}
-w''(x_N)+\left(\frac{x_N^2}{4}-\frac 12\right)w(x_N)=\Lambda(r)w(x_N)\quad\text{in}\quad(-\infty,r)\\
w(r)=0.
\end{cases}
$$
Therefore, the first Dirichlet eigenvalue of $-\Delta_\gamma$ coincides with the first eigenvalue of the one dimensional Schr\"odinger operator $\mathcal{H}_{V}$, where $V(\rho):=\frac{\rho^2}{4}-\frac 12$ is a convex function in $\R$. Since for any $r\in\R$ there exist two nonempty convex sets $C_0$, $C_1$ such that $H_r=\tau C_1+(1-\tau)C_0$, for some $\tau\in[0,1]$ (choose, for instance, $C_0=H_{\lfloor r \rfloor}$ and $C_1=H_{\lfloor r \rfloor+1}$) using Theorem \ref{th:lieb} we have that $\Lambda(r)=\lambda(\tau(r))$ is a convex function of $r\in\mathbb{R}$ with $\tau=\tau(r)$ given by $\tau(r)=r-\lfloor r\rfloor$. 


Since $\Lambda=g\circ\Phi$, we have that $g^{-1}=\Phi\circ\Lambda^{-1}$. Now $\Phi$ is smooth, and $\Lambda^{-1}$ is monotone decreasing and convex since $\Lambda$ is, and so $\Lambda^{-1}$ is locally Lipschitz. Therefore $g^{-1}$ is locally Lipschitz since it is composition of two locally Lipschitz functions.
\end{proof}

\section{Proof of the Main Theorem}
\label{sec:seztre}

Our strategy to prove the Main Theorem follows the ideas in \cite{BarBraJul, HanNad}: we first estimate $D_\gamma(\Omega)$ from below with a quantity involving the asymmetry of the superlevel sets of $u_\Omega$ and then, in a suitable range of values for the function $u_\Omega$, we show that the asymmetry of the superlevel sets is estimated from below by $\mathcal{A}_\gamma(\Omega)$.  From now on, $u_\Omega$ denotes the normalized nonnegative first eigenfunction for $\lambda_\gamma(\Omega)$.

The following proposition provides an enhanced version of an inequality proved in \cite[Theorem 3.1]{BeChFe}.
In the spirit of \cite{BraCinVit}, given a set $\Omega$, we exploit the sharp Gaussian quantitative isoperimetric inequality proved in \cite{BarBraJul} in order to estimate quantitatively the Gaussian perimeter of the level sets of $u_\Omega$.
\begin{proposition}
	\label{prop:quattroquattronostra}
	Let $\Omega\subset\R^N$ be an open set. For $t>0$, we set
	\begin{equation}\label{Def:mu}
	\Omega_{t}:=\left\{x\in\Omega: u_\Omega(x)>t \right\},\quad \mu(t):=\gamma(\Omega_{t}),
\end{equation}
	and, for any $m\in(0,1)$
	$$
	f(m):=\frac{e^{\frac{\Phi^{-1}(m)^2}{2}}}{1+\Phi^{-1}(m)^2}.
	$$
Then the function $\mu$ is absolutely continuous and for every halfspace $H$ s.t. $\gamma(H)=\gamma(\Omega)$ we have
	\begin{equation}
	D_\gamma(\Omega)=\lambda_\gamma(\Omega)-\lambda_\gamma(H)\ge\frac{1}{2c}
	\int_0^{\infty}f(\mu(t))\mathcal{A}^2_\gamma(\Omega_{t})\frac{I(\mu(t))}{-\mu^{'}(t)}dt,
	\end{equation}
	where c is the absolute constant in \cite[Main Theorem]{BarBraJul}. 
\end{proposition}
\begin{proof}
	By the coarea formula and thanks to the regularity of $u_\Omega$ we have that $\mu$ is absolutely continuous and also that
	\begin{equation}
	\label{eq:xder}
	\begin{split}
	\lambda_\gamma(\Omega)&=\int_{\Omega}|\nabla u_\Omega|^2d\gamma=\int_0^{\infty}dt\int_{\left\{u_\Omega=t\right\}}
	|\nabla u_\Omega|d\mathcal{H}_\gamma^{N-1}
	\\
	&\ge\int_0^{\infty}\frac{P_\gamma(\Omega_{t})^2}{\int_{\left\{u_\Omega=t\right\}}\frac{d\mathcal{H}_\gamma^{N-1}}{|\nabla u_\Omega|}}dt,
	\end{split}
	\end{equation}
	where we have used H\"older's inequality with exponents $(2,2)$ to get
	\begin{equation}\label{eq:holdfurba}
	P_\gamma(\Omega_{t})^2
	\le \left(\int_{\partial^\ast \Omega_{t}}|\nabla u_\Omega|\:d\mathcal{H}^{N-1}_\gamma\right)
	\left(\int_{\partial^\ast \Omega_{t}}\frac{d\mathcal{H}^{N-1}_\gamma}{|\nabla u_\Omega|}\right).
	\end{equation}
	We notice that the last integral in the right-hand side of \eqref{eq:holdfurba} is finite since $|\nabla u_\Omega|\ge \kappa_t>0$ on the level set $\partial^*\Omega_t$ for almost every $t\in(0,\|u_\Omega\|_\infty)$.

	Now, we consider the Ehrhard-symmetrized of the set $\Omega_{t}$
	$$
	\Omega_{t}^\ast=\left\{x\in\R^N:\quad u^\ast_\Omega(x)>t \right\}
	$$
	and, from the trivial inequality
	$$
	(P_\gamma(\Omega_{t})-P_\gamma(\Omega^\ast_{t}))^2\ge 0,
	$$
	we easily obtain 
	\begin{equation}
	\label{eq:square}
	P_\gamma(\Omega_{t})^2 \ge P_\gamma(\Omega^\ast_{t})^2+2P_\gamma(\Omega^\ast_{t})(P_\gamma(\Omega_{t})-P_\gamma(\Omega^\ast_{t})).
	\end{equation}
	By using the sharp quantitative Gaussian isoperimetric inequality \eqref{eq:cifumapr} we get
	\begin{equation}\label{eq:cifumapr2}
	P_\gamma(\Omega_t)-P_\gamma(\Omega_t^\ast)\ge 
	\frac{e^{\frac{r_t^2}{2}}}{4c(1+r_t^2)} \mathcal{A}_\gamma(\Omega_t)^2,
	\end{equation}
where $r_t$ is such that $\gamma(\Omega_t)=\Phi(r_t)$ and for some absolute constant $c>0$.
	Inserting \eqref{eq:cifumapr2} in \eqref{eq:square} we get 
	\begin{equation}
	\label{eq:quantsquared}
	P_\gamma(\Omega_{t})^2\ge P_\gamma(\Omega^\ast_{t})^2+\frac{f(\mu(t))}{2c}P_\gamma( \Omega^\ast_{t})\mathcal{A}_\gamma(\Omega_{t})^2.
	\end{equation}
From the equalities 
	$$
	\mu(t)=\gamma(\Omega^*_{t})=
	\int_t^{\infty}ds\int_{\partial \Omega^\ast_{s}}\frac{d\mathcal{H}^{N-1}_\gamma(x)}{|\nabla u_\Omega^\ast|},
	$$
	we deduce 
	\begin{equation}
	\label{eq:carlenkerce}
	\mu'(t)=-\int_{\partial \Omega^\ast_{t}}\frac{d\mathcal{H}^{N-1}_\gamma}{|\nabla u_\Omega^\ast|} \le -\int_{\partial \Omega_{t}}\frac{d\mathcal{H}^{N-1}_\gamma}{|\nabla u_\Omega|},
	\end{equation}
	where the inequality in \eqref{eq:carlenkerce} is proved in \cite[Lemma 4.3]{CarKer}. Inserting \eqref{eq:carlenkerce} and \eqref{eq:quantsquared} into \eqref{eq:xder} yields  
	\begin{equation}\label{eq:primaquantitativa}
	\lambda_\gamma(\Omega)\geq\int_0^{\infty}\frac{P_\gamma(\Omega^\ast_{t})^2}{-\mu^{'}(t)}dt
	+\frac{1}{2c}\int_0^{\infty}f(\mu(t))\frac{P_\gamma(\Omega^\ast_{t})\mathcal{A}_\gamma(\Omega_{t})^2}{-\mu^{'}(t)}dt. 
	\end{equation}	
	Using H\"older's inequality with exponents (2,2) as in \eqref{eq:holdfurba} and taking into account that the functions $|\nabla u_\Omega^\ast|^{1/2}$ and 
	$|\nabla u_\Omega^\ast|^{-1/2}$ are constant on the level plane $\partial \Omega^\ast_{t}$ we obtain  
	\begin{equation}\label{HolderEquality}
	\int_0^{\infty}\frac{P_\gamma(\Omega^\ast_{t})^2}{-\mu^{'}(t)}dt=\int_0^{\infty}\frac{P_\gamma(\Omega^\ast_{t})^2}{\int_{\partial \Omega^\ast_{t}}\frac{d\mathcal{H}^{N-1}_\gamma}{|\nabla u_\Omega^\ast|}}dt= \int_0^{\infty}\left(\int_{\partial \Omega^\ast_{t}}|\nabla u_\Omega^\ast|d\mathcal{H}^{N-1}_\gamma\right)dt. 
	\end{equation}
	By applying the coarea formula we get
	\begin{equation}\label{eq:coareaultima}
	\int_0^{\infty}\left(\int_{\partial \Omega^\ast_{t}}|\nabla u_\Omega^\ast|d\mathcal{H}^{N-1}_\gamma\right)dt=\int_{\Omega}|\nabla u_\Omega^\ast|^2d\gamma.
	\end{equation}
	By plugging \eqref{HolderEquality} and \eqref{eq:coareaultima} into \eqref{eq:primaquantitativa} we finally obtain
	\begin{align*}
	\lambda_\gamma(\Omega)=&\int_{\Omega}|\nabla u_\Omega|^2d\gamma\ge\int_{\Omega}|\nabla u^*_\Omega|^2 d\gamma+\frac{1}{2c}\int_0^{\infty}f(\mu(t))\frac{P_\gamma(\Omega^\ast_{t})\mathcal{A}_\gamma(\Omega_{t})^2}{-\mu^{'}(t)}dt
	\\
	\ge&\lambda_\gamma(H)+\frac{1}{2c}\int_0^{\infty}f(\mu(t))\frac{P_\gamma(\Omega^\ast_{t})\mathcal{A}_\gamma(\Omega_{t})^2}{-\mu^{'}(t)}dt,
	\end{align*}
	hence, recalling that $\gamma(H)=\gamma(\Omega)$ and $P_\gamma(\Omega^\ast_{t})=I(\gamma(\Omega^\ast_{t}))$, we get the thesis.
\end{proof}

The next lemma, proved in \cite[Lemma 4.2]{CCLP} (see also \cite[Lemma 2.8]{BraDeP} for a more general case) roughly says that if we know how asymmetric a set is and we consider another set which is not too different (in the measure sense) from the first one, then the asymmetry of the second set can be controlled from below by the asymmetry of the first one.
\begin{lemma}
	\label{lem:quattrounonostro}
	Let $E, F\subset\R^N$ be two measurable sets such that
	\begin{equation}
	\label{eq:transferofasym}
	\frac{\gamma(F\triangle E)}{\gamma(F)}\leq \kappa\mathcal{A}_\gamma(F),
	\end{equation}
	for some $0<\kappa<1/2$. Then
	$$
	\mathcal{A}_\gamma(E) \ge \frac{1-2\kappa}{c_\kappa}\mathcal{A}_\gamma(F),
	$$
	where $c_\kappa:=\begin{cases}\begin{array}{ll}
	1, &\text{if}\quad \gamma(E\setminus F)=0, \\
	1+2\kappa,\quad&\text{if}\quad \gamma(E\setminus F)>0.
	\end{array}\end{cases}$
\end{lemma}

Now our goal is to prove that
\begin{equation}
\label{eq:quantisoine}
D_\gamma(\Omega)=\lambda_\gamma(\Omega)-\lambda_\gamma(H)\ge C\mathcal{A}_\gamma(\Omega)^3,
\end{equation}
where $H$ is a halfspace such that $\gamma(H)=\gamma(\Omega)$.
We also observe that if $\lambda_\gamma(\Omega)\ge2\lambda_\gamma(H)$, then by using that $\mathcal{A}_\gamma(\Omega)<2$
$$
\lambda_\gamma(\Omega)-\lambda_\gamma(H)\ge\lambda_\gamma(H)>\lambda_\gamma(H)\frac{\mathcal{A}_\gamma(\Omega)^3}{8}.
$$
Therefore, we are reduced to considering the case
\begin{equation}
\label{eq:assumpdoubleperimeter}
\lambda_\gamma(\Omega)<2\lambda_\gamma(H).
\end{equation}
We are now ready to prove our quantitative Faber-Krahn inequality.

\begin{proof}[Proof of the Main Theorem]
Let us set 
$$
T:=\sup\left\{t>0\,:\,\gamma\left(\Omega_t\right)\ge\gamma(\Omega)\left(1-\frac 14\mathcal{A}_\gamma(\Omega)\right)\right\},
$$
which depends on the open set $\Omega$, and  
$$
T_0:=\frac{\beta}{4(1+\beta)}\mathcal{A}_\gamma(\Omega)\gamma(\Omega),
$$
for some $\beta>0$ that we choose in the sequel. Notice that $T_0<\frac 12$.

We suppose that $T\le T_0$ and we recall that $\Omega_T=\{u_\Omega>T\}$. Obviously, $\Omega_T$ is open since $u_\Omega$ is continuous in $\Omega$, and it is not empty. Indeed, from
$$
(u_\Omega-T)_+\ge u_\Omega-T,
$$
$\left\|u\right\|_{L^2(\Omega,\gamma)}=1$ and the Minkowski inequality, we deduce $\Omega_T$ has positive measure
\begin{equation}
\label{eq:nonempty}
\left\|(u_\Omega-T)_+\right\|_{L^2(\Omega_T,\gamma)}=\left\|(u_\Omega-T)_+\right\|_{L^2(\Omega,\gamma)}
\ge\left\|u\right\|_{L^2(\Omega,\gamma)}-T\sqrt{\gamma(\Omega)}\ge 1-T>0.
\end{equation}
As $(u_\Omega-T)_+$ is a competitor in the variational characterization \eqref{eq:1eigen} of $\lambda_\gamma(\Omega_T)$, we have 
\begin{equation}
\label{eq:fkra}
\lambda_\gamma(\Omega_T)\le\frac{\left\|\nabla(u_\Omega-T)_+\right\|^2_{L^2(\Omega_T,\gamma)}}{\left\|(u_\Omega-T)_+\right\|^2_{L^2(\Omega_T,\gamma)}}.
\end{equation}
From  
\begin{equation}
\label{eq:ottodec}
\left\|\nabla(u_\Omega-T)_+\right\|^2_{L^2(\Omega_T,\gamma)}\le\left\|\nabla u_\Omega\right\|^2_{L^2(\Omega,\gamma)}=\lambda_\gamma(\Omega),
\end{equation}
we infer
\begin{equation}
\label{eq:rtq}
\lambda_\gamma(\Omega)\ge\lambda_\gamma(\Omega_T)\left\|(u_\Omega-T)_+\right\|^2_{L^2(\Omega_T,\gamma)}\ge g\left(\gamma(\Omega_T)\right)\frac{\lambda_\gamma(H)}{g(\gamma(H))}\left\|(u_\Omega-T)_+\right\|^2_{L^2(\Omega_T,\gamma)},
\end{equation}
where in the first inequality we have used both \eqref{eq:fkra} and \eqref{eq:ottodec}, and in the second one we have exploited \eqref{eq:analyticfk}. 

By the definition of $T$ and the continuity of the application $[0,T]\ni t\mapsto \gamma(\Omega_t)\in(0,\gamma(\Omega)]$ we get $\gamma(\Omega_T)=\gamma(\Omega)\left(1-\frac 14\mathcal{A}_\gamma(\Omega)\right)$ where $\gamma(\Omega_T)\in\left(\frac12\gamma(\Omega),\gamma(\Omega)\right]$ since $\mathcal{A}_\gamma(\Omega)<2$. By using that $g$ is monotone decreasing and Proposition \ref{prop:bilipschitz} and denoting by $L_{\gamma(\Omega)}$ the biggest constant $L$ such that $g(a)-g(b)\geq L(b-a)$ for $a<b$ in the interval $\left(\frac12\gamma(\Omega),\gamma(\Omega)\right]$ we obtain
\begin{equation}
\label{eq:cfgh}
\begin{split}
g(\gamma(\Omega_T))&\ge g(\gamma(\Omega))+L_{\gamma(\Omega)}\left(\gamma(\Omega)-\gamma(\Omega_T)\right)\\
&= g(\gamma(\Omega))+L_{\gamma(\Omega)}\frac{\gamma(\Omega)}{4}\mathcal{A}_\gamma(\Omega).
\end{split}
\end{equation}
Inserting \eqref{eq:cfgh} in \eqref{eq:rtq} we have 
$$
\lambda_\gamma(\Omega)\ge\frac{\lambda_\gamma(H)}{g(\gamma(H))}\left(g(\gamma(\Omega))+L_{\gamma(\Omega)}\frac{\gamma(\Omega)}{4}\mathcal{A}_\gamma(\Omega)\right)\left\|(u_\Omega-T)_+\right\|^2_{L^2(\Omega_T,\gamma)}.
$$
Once we notice that
$$
\frac{g(\gamma(\Omega))}{g(\gamma(H))}=1
$$
and set 
$$
\frac{L_{\gamma(\Omega)}\gamma(\Omega)}{4 g(\gamma(H))}:=\beta>0,
$$
putting together the previous estimates we obtain
$$
\lambda_\gamma(\Omega)\ge\lambda_\gamma(H)(1+\beta\mathcal{A}_\gamma(\Omega))\left\|(u_\Omega-T)_+\right\|^2_{L^2(\Omega_T,\gamma)}.
$$
Using \eqref{eq:nonempty} and $\gamma(\Omega)<1$, we get 
$$
\left\|(u_\Omega-T)_+\right\|^2_{L^2(\Omega_T,\gamma)}\ge(1-T)^2\ge 1-2T_0\ge 1-\frac{\beta}{2(1+\beta)}\mathcal{A}_\gamma(\Omega),
$$
and so
$$
\lambda_\gamma(\Omega)\ge\lambda_\gamma(H)(1+\beta\mathcal{A}_\gamma(\Omega))\left(1-\frac{\beta}{2(1+\beta)}\mathcal{A}_\gamma(\Omega)\right),
$$
but since $\mathcal{A}_\gamma(\Omega)<2$ it is straightforward to see that
$$
(1+\beta\mathcal{A}_\gamma(\Omega))\left(1-\frac{\beta}{2(1+\beta)}\mathcal{A}_\gamma(\Omega)\right)\ge 1+\frac{\beta}{2(1+\beta)}\mathcal{A}_\gamma(\Omega),
$$
and this yields
$$
\lambda_\gamma(\Omega)-\lambda_\gamma(H)\ge\frac{\beta}{2(1+\beta)}\lambda_\gamma(H)\mathcal{A}_\gamma(\Omega)>\frac{\beta}{8(1+\beta)}\lambda_\gamma(H)\mathcal{A}_\gamma(\Omega)^3.
$$
Now we suppose that $T>T_0$. 
From Proposition \ref{prop:quattroquattronostra} and Lemma \ref{lem:quattrounonostro} (applied, for any $t\in[0,T]$, with $F=\Omega$, $E=\Omega_t$ and $\kappa=\frac14$) we get 
\begin{equation*}
\begin{split}
\lambda_\gamma(\Omega)-\lambda_\gamma(H)&\ge \frac{1}{2c}\int_0^{\infty}f(\mu(t))\mathcal{A}_\gamma(\Omega_{t})^2\frac{I(\mu(t))}{-\mu'(t)}dt \\
&\ge\frac{1}{2c}\int_{0}^{T}f(\mu(t))\mathcal{A}_\gamma(\Omega_{t})^2\frac{I(\mu(t))}{-\mu'(t)}dt \\
&\ge\frac{1}{2c}\cdot\frac14\mathcal{A}_\gamma(\Omega)^2\int_{0}^{T}f(\mu(t))\frac{I(\mu(t))}{-\mu'(t)}dt \\
&\ge\frac{\mathcal{A}_\gamma(\Omega)^2}{8c}\frac{e^{r^2/2}}{1+r^2}\int_{0}^{T}\frac{I(\mu(t))}{-\mu'(t)}dt \\
&\ge\frac{\mathcal{A}_\gamma(\Omega)^2}{8c}\frac{1}{1+r^2}\int_0^T\frac{dt}{-\mu'(t)},
\end{split}
\end{equation*}
where in the last two inequalities we respectively used the facts that $f(\mu(t))\geq \frac{e^{r^2/2}}{1+r^2}$ and $I(\mu(t))\ge e^{-r^2/2}$, where $r=\Phi^{-1}(\gamma(\Omega))$, since $\mu(t)\in\left(\frac12\gamma(\Omega),\gamma(\Omega)\right]$ for every $t\in\left[0,T\right]$.

 
This in turn implies that
\begin{equation}
\label{eq:jensen}
\lambda_{\gamma}(\Omega)-\lambda_{\gamma}(H)\ge\frac{\mathcal{A}_\gamma(\Omega)^2}{8c(1+r^2)}\int_{0}^{T}\frac{dt}{-\mu^{'}(t)}.
\end{equation}
We estimate the integral in the right-hand side of \eqref{eq:jensen} through Jensen's inequality
\begin{equation}
\label{eq:postjensen}
\int_{0}^{T}\frac{dt}{-\mu'(t)}\ge T^2
\left(\int_{0}^{T}-\mu^{'}(t)dt\right)^{-1}\ge
{T^2}\left(\gamma\left(\Omega\right)-\gamma\left(\Omega_{T}\right)\right)^{-1}=\frac{4T^2}{\gamma(\Omega)\mathcal{A}_\gamma(\Omega)},
\end{equation}
where in the last equality we used the definition of $T$. Summarizing, if we put \eqref{eq:postjensen} in \eqref{eq:jensen} we get
\begin{equation*}
\begin{split}
\lambda_{\gamma}(\Omega)-\lambda_{\gamma}(H)&\ge\frac{\mathcal{A}_\gamma(\Omega)^2}{8c(1+r^2)}\frac{4T^2}{\gamma(\Omega)\mathcal{A}_\gamma(\Omega)}\\
&=\frac{\mathcal{A}_\gamma(\Omega)}{2c(1+r^2) \gamma(\Omega)}T^2,
\end{split}
\end{equation*}
and recalling that
$$
T^{2}>(T_0)^{2}=\frac{C_\beta}{16}\mathcal{A}_\gamma(\Omega)^{2}\gamma(\Omega)^{2},
$$
we conclude that
\begin{equation}
\label{eq:conclusion}
\lambda_{\gamma}(\Omega)-\lambda_{\gamma}(H)\ge\frac{\gamma(\Omega)C_\beta}{32 c(1+r^2)}\mathcal{A}^3_\gamma(\Omega),
\end{equation}
where $C_\beta:=\left(\frac{\beta}{\beta+1}\right)^{2}$.
\end{proof}

\end{document}